\documentclass[12pt]{article}
\usepackage{fullpage}
\usepackage{amsfonts, amsmath, amssymb,latexsym,epsfig}
\usepackage{amsthm}
\usepackage{tikz}
\usetikzlibrary{graphs}
\usetikzlibrary{graphs.standard}
\usepackage{relsize}
\usepackage{verbatim}
\usepackage{enumerate}
\usepackage[skip=10pt plus1pt, indent=20pt]{parskip}
\usepackage{bbding, pifont}

\usepackage{hyperref}
\usepackage{cleveref}

\newtheorem{thm}{Theorem}[section]
\newtheorem{cor}[thm]{Corollary}
\newtheorem{lem}[thm]{Lemma}
\newtheorem{remark}[thm]{Remark}
\newtheorem{prop}[thm]{Proposition}

\newtheorem{ex}[thm]{Example}

\theoremstyle{definition}
\newtheorem{notation}[thm]{Notation}

\newcommand{\uu}{\mathbf{u}}
\newcommand{\vv}{\mathbf{v}}

\newcommand{\Z}{\mathbb Z}

\DeclareMathOperator{\PG}{PG}
\newcommand{\CC}{\mathcal C}
\newcommand{\rowsp}{\mathrm{rowsp} }
\newcommand{\rk}{\mathrm{rk}}
\newcommand{\qbin}[2]{\genfrac{[}{]}{0pt}{}{#1}{#2}}

\newcommand{\F}{\mathbb F}
\newcommand{\dH}{\mathrm{d}_{\mathrm{H}}}
\newcommand{\wH}{\mathrm{w}_{\mathrm{H}}}

\newcommand{\paolo}[1]{{\color{violet} \sf $\star\star$ Paolo: [#1]}}
\newcommand{\jozefien}[1]{{\color{magenta} \sf $\star\star$ Jozefien: [#1]}}
\newcommand{\vlad}[1]{{\color{orange} \sf $\star\star$ Vlad: [#1]}}

\author{
Jozefien D'haeseleer\thanks{Department of Mathematics: Analysis, Logic and Discrete Mathematics, Ghent University, 9000 Ghent, Belgium. 
E-mail: \texttt{jozefien.dhaeseleer@ugent.be}, \texttt{vlad.taranchuk@ugent.be}.}
\and
Francesco Pavese\thanks{Dipartimento di Meccanica, Matematica e Management, Politecnico di Bari, Via Orabona 4, 70125 Bari, Italy. 
E-mail: \texttt{francesco.pavese@poliba.it}, \texttt{paolo.santonastaso@poliba.it}.}
\and
Paolo Santonastaso\footnotemark[2]
\and
Vladislav Taranchuk\footnotemark[1]
}
\title{Chromatic Number of Grassmann Graphs and MRD codes}

\begin{document}

\maketitle

\begin{abstract}

In this paper we investigate the chromatic number of the Grassmann graphs and of their powers, denoted $J_q(n,m,t)$. In this graph, the vertices correspond to the $m$-dimensional subspaces in $\mathbb{F}_q^n$ and two vertices are adjacent if the corresponding subspaces intersect in a subspace of dimension at least $t$. 

By generalizing the lifting technique of Silva, K\"otter and Kschischang, we use \emph{maximum rank distance (MRD)} codes to establish that \[\chi(J_q(n, m, t)) \leq (1 +o(1))n^{m-t}q^{(n-m)(m-t)})\]
when $n \geq 2m$. Given that $J_q(n, m, t)$ is isomorphic to  $J_q(n,n-m,n-2m+t)$, this establishes a new upper bound on $J_q(n, m, t)$ for any valid choice of parameters. Furthermore, we observe that in the regime that $n, m $, and $t$ are fixed, our bound is asymptotically tight, implying that 
$$
\chi(J_q(n, m, t)) = \Theta(q^{(m-t)\max(n-m, m)}).
$$
\end{abstract}

\section{Introduction}

Let $G$ be a graph. We denote by $V(G)$ the vertex set of $G$, and by $E(G)$ the edge set of $G$. A \emph{clique} of $G$ is a set of vertices of $G$ in which every two distinct vertices are adjacent.  
On the other hand, an \emph{independent set} of $G$ is a set of vertices of $G$ no two of which are adjacent.

The \emph{chromatic number} of $G$, denoted by $\chi(G)$, is the smallest integer $c$ such that $V(G)$ can be partitioned into $c$ independent sets. In other words, it corresponds to the minimum number of colours required to colour its vertices such that no two adjacent vertices share the same colour.

The \emph{degree} of a vertex is the number of edges incident to it. 
If $\Delta$ denotes the maximum degree of $G$, that is, the largest degree among its vertices, then it is well known that 
\begin{equation} \label{eq:chromaticlessdegree}
\chi(G) \leq \Delta + 1.
\end{equation}

Starting from a given graph, one can construct a new related graph through the notion of \emph{graph powers}.  
More precisely, if $G$ is a graph, the distance between two vertices 
is the length (number of edges) of a shortest path connecting them.  
The \emph{$k$-th power} of $G$ 
is the graph with the same vertex set $V(G)$ and two vertices are adjacent if and only if their distance in $G$ is at most $k$.  

Let $n > m$.  
The Johnson graph $J(n,m)$ is the graph whose vertex set consists of all $m$-subsets of  
$[n] := \{1,2,\ldots,n\}$, where two vertices are adjacent whenever their intersection has cardinality $m-1$.  
It is well-known \cite{GS} that 
\[
n - m + 1   \le   \chi(J(n,m))   \le   n.
\]

In this paper we study the $q$-analogue of the Johnson graph, namely the \emph{Grassmann graph} $J_q(n,m)$. 
For this, we first briefly mention the notation used throughout this work.
\begin{notation}
    Let $q$ denote a prime power. 
We write $\F_q$ for the finite field with $q$ elements, and $\F_q^n$ for the $n$-dimensional vector space over $\F_q$. 
The projective space of dimension $n$ over $\F_q$ will be denoted by $\PG(n,q)$. 
Recall that a projective subspace of dimension $k$ corresponds to a vector subspace of dimension $k+1$ in $\F_q^{n+1}$. 
\end{notation}
The Gaussian binomial coefficient is defined by
\[
\qbin{n}{m}_q = \frac{(q^n - 1)(q^{n-1}-1)\cdots(q^{ n-m+1}-1)}{(q^m-1)(q^{m-1}-1)\cdots(q-1)}.
\]
Recall that $\qbin{n}{m}_q$ counts the number of $m$-dimensional subspaces of $\F_q^n$, 
or equivalently, the number of $(m-1)$-dimensional subspaces of $\PG(n-1,q)$.
For two vectors $\mathbf{u} = (u_1,\ldots,u_n)$ and $\mathbf{v} = (v_1,\ldots,v_n)$ in $\F_q^n$,  
the \emph{Hamming distance} between $\mathbf{u}$ and $\mathbf{v}$ is denoted by $\dH(\mathbf{u},\mathbf{v})$.  
The \emph{Hamming weight} of $\mathbf{u}$ is defined as 
$
\wH(\mathbf{u}) := \dH(\mathbf{u},\mathbf{0})$. The \emph{Schur product} of $\mathbf{u}$ and $\mathbf{v}$ is defined as $
\mathbf{u} \star \mathbf{v} = (u_1 v_1, \ldots, u_n v_n)$. There is a well-known relation connecting Hamming distance, Hamming weight, and the Schur product, namely
\begin{equation} \label{eq:generalrelationdHammingshur}
\dH(\mathbf{u},\mathbf{v})
= \wH(\mathbf{u})
+ \wH(\mathbf{v})
- 2 \wH(\mathbf{u} \star \mathbf{v}).
\end{equation}
\\
With this notation, the Grassmann graph can be defined as follows.

The vertices of $J_q(n,m)$ are the $m$-dimensional subspaces of $\F_q^n$, and two vertices are adjacent if their intersection has dimension $m-1$.  Recently, the first general upper bound on $\chi(J_q(n, m))$ appeared which improves upon the trivial $\Delta + 1$ bound. These bounds are analogous to the classical bounds for Johnson graphs.

\begin{thm}[\textnormal{\cite[Theorem 1.2]{dhaeseleer2025chromatic}}]
Let $q$ be a prime power and $m<n$. Then 
\[
\qbin{n-m+1}{1}_q   \le   \chi(J_q(n,m))   \le   \qbin{n}{1}_q.
\]
\end{thm}

When $m=2$, one recovers the line-incidence graph of the projective space $\PG(n-1,q)$.  
In particular, a colouring of $J_q(n,2)$ using $\qbin{n-1}{1}_q$ colours corresponds to a partition of the lines of $\PG(n-1,q)$ into line-spreads;  
such a partition is called a \emph{line-parallelism}, and the existence of line-parallelisms has long been a central topic in finite geometry \cite{Baker76, Beutelspacher74, Denniston1973, Heering_parallelism_PGn2, Johnson_survey, Johnson2010,   Pavese_parallelism_PG3q, parallelisms_1973}. The existence of line-parallelisms has been settled for various parameter sets, and hence, it is known that 
$$
\chi(J_q(n, 2)) = \qbin{n-1}{1}_q \text{ when } \left\{ \begin{array}{c}
     n = 4 \text{ and any prime power }q   \\
     q = 2, 3, 4, 8, 16 \text{ and any even }n 
\end{array}\right.
$$
When $n$ is odd and $q = 2$, Meszka \cite{Meszka2013} proved that 
$$
\chi(J_2(n, 2)) = \qbin{n-1}{1}_2 + 3.
$$
The only other cases for which there is an improvement over the general upper bound in Theorem 1.2 is again the case $m=2$, when $q=2^e$ and $n$ is even \cite{dhaeseleer2025chromatic}. In particular, for these parameters, it is shown that 
$$
\chi(J_q(n, 2)) < 2\qbin{n-1}{1}_q.
$$

\medskip

In this paper, we investigate the chromatic number of Grassmann graphs and of their powers.  
This work fits into the line of research concerned with determining the chromatic number of powers of graphs, 
a topic that dates back to the seminal papers by Kramer and Kramer in 1969, see \cite{kramer1969farbungsproblem,kramer1969probleme}. Over the years, particular attention has been devoted to this number, 
with special emphasis on the square of graph; see, for instance, \cite{borodin2012list,havet2009choosability,havet2007list,thomassen2018square}.  
More recently, in \cite{abiad2025eigenvalue}, new spectral bounds for the chromatic number of powers of graphs were obtained, 
leading to applications in coding theory.

In this paper, we establish a new upper bound on the chromatic number of the Grassmann graphs as well as their powers. For any $t \in \{1,\ldots,m-1\}$, we define  
\[
J_q(n,m,t)   
\]
to be the $(m-t)$-th power of $J_q(n, m)$. Hence, $J_q(n, m, t)$ is the graph whose vertices are all the $m$-dimensional subspaces of $\F_q^n$,  
and two vertices are adjacent if and only if their intersection has dimension at least $t$. We note that $J_q(n, m, t)$ is regular of degree $\sum_{i=t}^{m-1} \qbin{m}{i}_q\qbin{n-m}{m-i}_qq^{(m-i)^2}$. In this sum, the term $\alpha_i= \qbin{m}{i}_q\qbin{n-m}{m-i}_qq^{(m-i)^2}$ is the number of $m$-spaces meeting a fixed $m$-space in a subspace of dimension $i$, see \cite[Lemma 9.3.2]{BCN}.
For $t=m-1$ we recover the classical Grassmann graph, i.e.,
\[
J_q(n,m,m-1) = 
J_q(n,m).
\]

Our upper bound will be established via an explicit colouring that relies on optimal rank-metric codes. Codes endowed with the rank metric have gained substantial attention over the past decades due to their rich algebraic structure and numerous applications. The modern interest in this area was sparked by the groundbreaking work of Silva, K\"otter and Kschischang  \cite{silva2008rank}, where they employed rank-metric codes as tools in linear random network coding. However, the origins of the subject go back to seminal paper by Delsarte \cite{delsarte1978bilinear}, where rank-metric codes were first introduced from a combinatorial perspective, and to Gabidulin’s independent rediscovery in \cite{ga85a}. 
 We refer the reader to \cite{bartz2022rank, gorla2018codes} for a comprehensive introduction and an overview of their most significant applications. A rank-metric code can be considered a subset of the metric space $(M_{m\times h}(\F_q),\rk)$, where $M_{m\times h}(\F_q)$ denotes the set of $m \times h$ matrices over $\F_q$ and $\rk$ is the matrix rank.
Among them, a particularly important class is that of \emph{maximum rank distance (MRD)} codes. These codes attain the \emph{Singleton-like bound} of Delsarte and therefore have optimal parameters:  
for a prescribed minimum rank distance, they achieve the largest possible cardinality.

\subsection*{Our contribution}
We begin our investigation of $
\chi(J_q(n, m, t))$ by studying the basic structural properties of the graphs.  
In particular, by using the duality symmetry
\[
J_q(n,m,t)  \cong   J_q(n,n-m,n-2m+t),
\]
 we can reduce the analysis of $\chi(J_q(n,m,t))$ to two complementary parameter regimes,
namely $n \ge 2m$ and $m < n < 2m$. 

\medskip
\noindent
The main contribution of this paper is a new colouring technique based on 
\emph{maximum rank distance (MRD)} codes.  
Generalising the lifting construction of Silva, K\"otter and Kschischang, 
we associate to each identifying vector $\mathbf{u}$ of weight $m$ a family of 
$m$-dimensional subspaces obtained from disjoint cosets of an MRD code.  
We show that each such family forms a coclique in $J_q(n,m,t)$, and that different cosets
produce disjoint cocliques.  
This yields an explicit colouring of all vertices with the same identifying vector.

When $n \ge 2m$, this MRD-based construction provides a proper colouring of the entire graph.  
Combining it with a partition of identifying vectors arising from the powers of the Johnson graph $J(n,m)$
gives an upper bound on $\chi(J_q(n,m,t))$ in this regime.  
Using the duality isomorphism, we extend the argument to the complementary range $m < n < 2m$. 

\begin{thm}\label{MainThm}
    Let $q$ be a prime power and $n, m$ be positive integers satisfying $n \geq 2m$. Then 
    $$
    \qbin{n-t}{m-t}_q\leq \chi(J_q(n, m, t)) \leq (1 + o(1))n^{m-t} \qbin{n-t}{m-t}_q.
    $$
    When instead we have $m \leq n < 2m$, then 
    $$
    \qbin{2m-t}{ m-t }_q 
    \le 
    \chi(J_q(n,m,t))
    \le 
    (1 + o(1))n^{m-t} \qbin{2m-t}{ m-t }_q 
    $$
\end{thm}

In particular, we highlight that for a fixed $n, m, t$, our result asymptotically determines $\chi(J_q(n, m, t))$ as a function of $q$, implying that 
$$
\chi(J_q(n, m, t)) = \Theta(q^{(m-t)\max(n-m, m)})
$$

The remainder of the paper is organized as follows: In section 2, we introduce the Johnson graphs, their powers, and we collect all known results regarding the chromatic number of these graphs. We briefly mention that our own upper bound $\chi(J_q(n, m , t))$ relies on the the chromatic number of the Johnson graphs, hence we include a short survey of what is already known regarding the chromatic number of these graphs. In section 3, we establish a lower bound on $\chi(J_q(n, m, t))$. In section 4, we give the explicit colouring of $J_q(n, m, t)$ using MRD codes. In section 5, we give some concluding remarks and open questions.

\section{On the chromatic number of the powers of the Johnson graph}
Let $n > m$.  
The \emph{Johnson graph} $J(n,m)$ is the graph whose vertex set consists of all $m$-subsets of 
$[n] := \{1, 2, \dots, n\}$, where two vertices are adjacent whenever their intersection has cardinality $m-1$. We now consider the powers of the Johnson graph.  
For any $t \in \{1, \ldots, m-1\}$, we define
\[
J(n,m,t) 
\]
to be $(m-t)$-th power of $J(n, m)$. Hence, the vertices of  $J(n,m,t)$ are all the $m$-subsets of $[n]$, 
and two vertices are adjacent if their intersection has cardinality at least $t$.  
Clearly, for $t = m-1$, we recover the original Johnson graph:
\[
J(n,m,m-1) = 
J(n,m).
\]

\noindent The study of the chromatic number of $J(n,m,t)$ is closely related to the problem of determining 
the minimum and maximum sizes of constant-weight codes in the Hamming metric.  
We recall this well-known correspondence below and derive the asymptotic order of the chromatic number $\chi(J(n,m,t))$.  
Although these results are classical, we briefly describe and collect them here for completeness, 
since to the best of our knowledge there is no single reference that presents them in a unified form. 

\noindent To any subset $S \subseteq [n]$, we associate its characteristic vector 
$\mathbf{v}_S \in \mathbb{F}_2^n$. 
The $i$-th coordinate of $\mathbf{v}_S$ is equal to $1$ precisely when $i \in S$, 
and equal to $0$ otherwise.
Thus every $m$-subset $S$ of $[n]$ corresponds to a binary vector of
Hamming weight $m$. For, two $m$-subsets $S,T$ of $[n]$, the Shur product $\mathbf{v}_S \star \mathbf{v}_T$ has a $1$
exactly in the coordinates belonging to $S \cap T$, hence
$|S \cap T| = \wH(\mathbf{v}_S \star \mathbf{v}_T)$.
As a consequence of Eq. \eqref{eq:generalrelationdHammingshur}, we have
\[
\dH(\mathbf{v}_S,\mathbf{v}_T)
= \wH(\mathbf{v}_S)+\wH(\mathbf{v}_T)
  - 2 \wH(\mathbf{v}_S \star \mathbf{v}_T),
\]
and since both vectors have weight $m$, we obtain
\[
\dH(\mathbf{v}_S,\mathbf{v}_T)
= 2m - 2|S \cap T|.
\]

Now consider the power of the Johnson graph $J(n,m,t)$.
Hence an independent set $\mathcal{C}$ in $J(n,m,t)$ is precisely a family of
$m$-subsets satisfying
$
|S \cap T| \le t-1, 
$
for all distinct  $S,T\in\mathcal{C}$, which is equivalent to
\[
\dH(\mathbf{v}_S,\mathbf{v}_T) \ge 2(m-t+1).
\]

Therefore, the independence number $\alpha(J(n,m,t))$, i.e. the size of the largest independent set in $J(n,m,t)$, is exactly the size $A(n,m,2(m-t+1))$ of the largest constant-weight binary code of length $n$, weight $m$ and minimum distance $2(m-t+1)$:
\[
\alpha\left(J(n,m,t)\right)
= A\big(n, m,  2(m-t+1)\big).
\]

Since every proper colouring of $J(n,m,t)$ partitions the 
$\binom{n}{m}$ vertices into independent sets, we have the general
bound
\[
\chi\left(J(n,m,t)\right)
\ge
\frac{\binom{n}{m}}{\alpha\left(J(n,m,t)\right)}
=
\frac{\binom{n}{m}}{A(n,m,2(m-t+1))}.
\]

We can now apply the Johnson bound for binary constant–weight codes  \cite[Section 3, Corollary 1]{GS}.
\[
A(n,m,2(m-t+1)) \leq \frac{\binom{n}{m-(m-t+1)+1}}{\binom{m}{m-(m-t+1)+1}}=\frac{\binom{n}{t}}{\binom{m}{t}}\lesssim \frac{(m-t)!n^{t}}{m!} ,
\] 
as $n \rightarrow  \infty$.
Therefore, we obtain a lower bound on the
chromatic number of the power of the Johnson graph \(J(n,m,t)\):
\begin{equation} \label{eq:lowegpowerjohnson}
\chi\left(J(n,m,t)\right)
\ge
\frac{ \binom{n}{m}}
{A(n,m,2(m-t+1))} \gtrsim \frac{n^m}{m!} \cdot \frac{m!}{(m-t)!n^{t}}  \gtrsim \frac{n^{m-t}}{(m-t)!}, 
\end{equation}
as $n \rightarrow  \infty$.

To complement this with an upper bound on \(\chi(J(n,m,t))\), we use an explicit colouring
method which can be found in Graham and Sloane \cite[Section II]{GS}.
Let $p$ be the smallest prime satisfying $p \ge n + 1$. We note that the work of Baker, Harman, and Pintz \cite{BakerHarmanPintz2001} also implies that $n < p < n + n^{0.525} $ for all sufficiently large $n$. Let $r=(p^{m-t + 1} - 1)/(p-1)$. Let $B_{m-t}$ be a Bose-Chowla set in $\Z_r$, that is, the sum of any $m-t$ elements of $B_{m-t}$ is distinct. Such sets exist for any $r$ of the given form and have size $p+1$. Now, let $\Phi:[n] \rightarrow B_{m-t}$ be any one-to-one mapping. To a subset $S$ of size $m$ we assign the colour
\[
   \bar{\Phi}(S)
   :=
       \sum\limits_{a \in S} \Phi(a) \pmod r \in \mathbb{Z}_r.
\]
As proved in \cite[Theorem 9]{GS}, for any element $j \in \mathbb{Z}_r$, $\bar{\Phi}^{-1}(j)$ corresponds to a binary constant–weight
code in $\mathbb{F}_2^n$ of weight $m$ and minimum distance at least $2(m-t+1)$.
Consequently, adjacent vertices in $J(n,m,t)$ always receive distinct colours.
Hence the map $S\mapsto \bar{\Phi}(S)$ is a proper colouring of $J(n,m,t)$, and
\begin{equation} \label{eq:ordergeneralizedJohnson}
   \chi\left(J(n,m,t)\right)
   \le
   r  \leq (1 + o(1))p^{m-t} = (1 + o(1))n^{m - t}
\end{equation}

\noindent Combining the lower bound Eq. \eqref{eq:lowegpowerjohnson} with this
Graham–Sloane upper bound, we conclude that
\[
   \Omega \left( \frac{n^{m-t}}{(m-t)!} \right)
    \le 
   \chi \left(J(n,m,t)\right)
    \le 
    (1 + o(1))n^{m-t},
\]
so, up to multiplicative constants depending only on \(m\) and \(t\),
the chromatic number of \(J(n,m,t)\) grows polynomially in \(n\) with
degree exactly \(m-t\). In particular, this provides polynomial upper and lower bounds on the
growth of the chromatic number of the power of the Johnson graph
\(J(n,m,t)\) as \(n \to \infty\).

\section{Lower bounds on $\chi(J_q(n, m ,t))$}

We begin by investigating some basic properties of $J_q(n,m,t)$. Recall that two graphs $G$ and $G'$ are said to be \emph{isomorphic} if there exists a bijection 
$
\varphi : V(G) \longrightarrow V(G')
$
such that two vertices $x,y \in V(G)$ are adjacent in $G$ if and only if $\varphi(x)$ and $\varphi(y)$ are adjacent in $G'$. We recall that the Grassmann graph $J_q(n,m)$ is isomorphic to the Grassmann graph $J_q(n,n-m)$. As a consequence, the $(m-t)$-th power of the Grassmann graph $J_q(n,m,t)$ is isomorphic to the $(m-t)$-th power of the Grassmann graph $J_q(n,n-m)$, i.e. $J_q(n,n-m,n-2m+t)$ and so
\begin{equation}\label{Isomorphism}
    J_q(n,m,t) \cong J_q(n,n-m,n-2m+t)
\end{equation}
An explicit isomorphism can de obtained as follows. 
To any subspace $U$ of $\F_q^n$, we can associate its (orthogonal) dual subspace, defined by
\[
U^{\perp} := \{ (w_1,\ldots,w_n) \in \F_q^n : \sum_{i=1}^nw_iu_i = 0 \text{ for all } (u_1,\ldots,u_n) \in U  \}.
\]
Then the map $\varphi :V(J_q(n,m,t)) \to V(J_q(n,n-m,n-2m+t))$, defined by $\varphi(U) = U^{\perp}$ is a graph isomorphism.

We now turn to bounds for the chromatic number $\chi(J_q(n,m,t))$. 

\begin{prop}
We have
\[
 \chi(J_q(n,m,t)) \geq 
 \max \left\{ \qbin{n-t}{ m-t }_q,  \qbin{2m-t}{ m-t }_q \right\}
 =
 \begin{cases}
   \qbin{n-t}{ m-t }_q & \text{if } n \geq 2m, \\[6pt]
   \qbin{2m-t}{ m-t }_q & \text{if } 2m > n > m.
 \end{cases}
\]
\end{prop}

\begin{proof}
Observe that $J_q(n,m,t)$ contains a clique of size 
\[
 \qbin{n-t}{ m-t }_q,
\]
since this is exactly the number of $m$-dimensional subspaces containing a fixed $t$-dimensional subspace. 
Hence,
\[
 \chi(J_q(n,m,t)) \geq \qbin{n-t}{ m-t }_q,
\]
because all vertices of such a clique must receive distinct colours. Moreover, by Eq. \eqref{Isomorphism}, we know that 
\[
 \chi(J_q(n,m,t)) = \chi\big(J_q(n,n-m,n-2m+t)\big).
\]
Applying the same argument to $J_q(n,n-m,n-2m+t)$ yields
\[
 \chi(J_q(n,m,t)) \geq \qbin{2m-t}{ m-t }_q.
\]
Combining these two bounds gives
\[
 \chi(J_q(n,m,t)) \geq 
 \max \left\{ \qbin{n-t}{ m-t }_q,  \qbin{2m-t}{ m-t }_q \right\},
\]
as claimed.
\end{proof}

\section{Upper bounds on $\chi(J_q(n, m,t))$ via MRD codes}

We begin by briefly recalling the necessary results on rank-metric codes. Let $M_{m\times h}(\F_q)$ denote the set of $m \times h$ matrices over $\F_q$. 
This space can be endowed with the \emph{rank metric}, defined by
\[
d(A,B) = \mathrm{rk}(A-B),
\]
for every $A,B \in M_{m\times h}(\F_q)$.
A \emph{(linear) rank-metric code} $\CC$ is an $\F_q$-subspace of $M_{m\times h}(\F_q)$ equipped with the rank metric. 
The minimum distance of $\CC$ is defined as
\[
d=d(\CC) = \min\{  d(A,B) : A,B \in \CC,  A \neq B  \}.
\]

Delsarte showed in \cite{delsarte1978bilinear} that the parameters of these codes must satisfy the \emph{Singleton-like bound}:
\begin{equation}\label{eq:Singleton} 
|\CC| \leq q^{\max\{m,h\} (\min\{m,h\}-d+1)}. 
\end{equation}
When equality holds in Eq. \eqref{eq:Singleton}, the code $\CC$ is called a \emph{maximum rank distance} (MRD) code.

For any admissible choice of $q,m,h$ and $d$, there exists a rank-metric
code $\CC$ in $M_{m\times h}(\F_q)$ that attains the Singleton-like bound. 
This result was first proved by Delsarte in \cite{delsarte1978bilinear}, and later rediscovered independently by Gabidulin in \cite{ga85a}.

\begin{thm} \label{thm:existenceMRD}
For all $1 \leq d \leq \min\{m,h\}$, there exists an MRD code 
$\CC$ in $M_{m\times h}(\F_q)$ with minimum distance $d(\CC) = d$ and 
\[
|\CC| = q^{\max\{m,h\} (\min\{m,h\}-d+1)}.
\]
\end{thm}

Our explicit colouring of $\chi(J_q(n, m, t))$ relies entirely on this constructive existence result of MRD codes. The key idea is to construct $m$-dimensional subspaces of $\F_q^{n}$ starting from the matrices of disjoint cosets of an MRD code.  
A similar approach can be traced back to the groundbreaking work of Silva, K\"otter and Kschischang \cite{silva2008rank}, 
where rank-metric codes were proposed as a tool for linear random network coding.  

Let $n>m$. Let $I_m$ denote the identity matrix in $M_{m\times m}(\F_q)$. Let $\mathbf{u} = (u_1,\ldots,u_{n}) \in \F_2^{n}$ be a binary vector of Hamming weight $m$, 
and let $A \in M_{m\times n-m}(\F_q)$.  
Let 
\[
r_1 < r_2 < \cdots < r_m \quad \text{and} \quad s_1 < s_2 < \cdots < s_{n-m}
\]
be the indices of the entries of $\mathbf{u}$ equal to $1$ and $0$, respectively.  
We define the matrix \[L_{\mathbf{u}}(A) \in M_{m\times n}(\F_q)\]
as the matrix whose $j$-th column is given, for $j \in \{1,\ldots,n\}$, by
\[
\begin{cases}
\text{the $i$-th column of } I_m, & \text{if } j = r_i, \\[4pt]
\text{the $k$-th column of } A, & \text{if } j = s_k.
\end{cases}
\]

\noindent In other words, $L_{\mathbf{u}}(A)$ is obtained by inserting the columns of the identity matrix $I_m$ 
into the positions of $\F_2^{n}$ indicated by the $1$-entries of $\mathbf{u}$, 
and filling the remaining positions (corresponding to the $0$-entries) with the columns of $A$, in order. For our purposes, we restrict the domain of $L_\mathbf{u}$ only to those $A$ for which $L_\mathbf{u}(A)$ is in reduced-row echelon form.

\begin{ex}
Let $n=5$, $m=2$, and $\mathbf{u} = (1,0,1,0,0)$.  
Then the positions of the $1$'s are $r_1=1$, $r_2=3$, and the positions of the $0$'s are $s_1=2$, $s_2=4$, $s_3=5$.  
For
\[
A =
\begin{bmatrix}
a_{11} & a_{12} & a_{13} \\
0 & a_{22} & a_{23}
\end{bmatrix},
\]
the matrix $L_{\mathbf{u}}(A)$ is given by
\[
L_{\mathbf{u}}(A) =
\begin{bmatrix}
1 & a_{11} & 0 & a_{12} & a_{13} \\
0 & 0 & 1 & a_{22} & a_{23}
\end{bmatrix}.
\]
Hence, the columns of $I_2$ appear in the positions of the $1$'s of $\mathbf{u}$, 
and the columns of $A$ occupy the remaining positions in order.
\end{ex}



\begin{remark}
    Note that if we go from matrices to subspaces, then all columns $s_i$ with $s_i<r_m$ should have zeroes in the last $m-\max\{l|r_l<s_i\}$ rows because of our restriction that the matrix $L_\mathbf{u}(A)$ must be in row echelon form. In particular, $L_u(A)$ is a full rank RREF matrix. 
    In our arguments, we give a colouring of all full rank RREF matrices. We note that these matrices are in one-to-one correspondence with the set of all subspaces of dimension $m$.
\end{remark}
 
In the special case where $\mathbf{u} = (\underbrace{1,\ldots,1}_{m\ \text{times}},\underbrace{0,\ldots,0}_{n-m\ \text{times}}) \in \F_2^{n}$,  
the matrix $L_{\mathbf{u}}(A)$ reduces to
$
[I_m \  A],
$
that is, the $m \times n$ matrix obtained by concatenating $I_m$ and $A$.

For a vector $\mathbf{u} \in \F_2^{n}$ and a matrix $A\in   M_{m \times n-m}(\F_q)$, define the $m$-dimensional subspace
\[
\Lambda_{\mathbf{u}}(A) = \rowsp(L_{\mathbf{u}}(A)) \subseteq \F_q^{n},
\]
where $\rowsp(M)$ denotes the row space of a matrix $M$, and it is called the $\mathbf{u}$-\emph{lifting} of $A$.  
For a subset $\mathcal{D} \subseteq M_{m \times n-m}(\F_q)$, we set the $\mathbf{u}$-\emph{lifting} of $\mathcal{D}$ as
\[
\Lambda_{\uu}(\mathcal{D}) = \{\Lambda_{\uu}(A) : A \in \mathcal{D}\}.
\]

It is easy to verify that, for every binary vector $\mathbf{u} \in \F_2^{n}$ of Hamming weight $m$, 
the mapping 
\[
A \longmapsto \Lambda_{\mathbf{u}}(A)
\]
is injective. In particular, $|\Lambda_{\uu}(\mathcal{D})| = |\mathcal{D}|$, for any $\mathcal{D} \subseteq M_{m \times n-m}(\F_q)$. 

The dimension of the intersection of the lifting of two matrices 
can be computed as follows. This is a reformulation of \cite[Proposition 4]{silva2008rank}, 
but we include a proof for completeness.

\begin{lem} \label{lm:intersectionlifting}
Let $\mathbf{u} \in \F_2^{n}$ be a binary vector of Hamming weight $m$, 
and let $A,B \in M_{m\times n-m}(\F_q)$. Then
\[
\dim\big(\Lambda_{\mathbf{u}}(A) \cap \Lambda_{\mathbf{u}}(B)\big) = m - \rk(A-B).
\]
\end{lem}

\begin{proof}
Let $\mathbf{v} = (\underbrace{1,\ldots,1}_{m \text{ times}},0,\ldots,0) \in \F_2^{n}$.  Since $\mathbf{u}$ has Hamming weight $m$, there exists a permutation matrix $P \in M_{n\times n}(\F_q)$ 
that maps the positions of the $1$'s in $\mathbf{v}$ to those in $\mathbf{u}$, 
so that
\[
L_{\mathbf{u}}(A) = L_{\mathbf{v}}(A) P.
\]
The multiplication by $P$ corresponds to a reordering of the coordinates of $\F_q^{n}$, 
which does not affect the dimension of the intersection of subspaces.  
Hence,
\[
\dim\big(\Lambda_{\mathbf{u}}(A) \cap \Lambda_{\mathbf{u}}(B)\big)
  = \dim\big(\Lambda_{\mathbf{v}}(A) \cap \Lambda_{\mathbf{v}}(B)\big).
\]
We can therefore compute the intersection dimension for $\mathbf{v}$ instead.  
Since $\Lambda_{\mathbf{v}}(A) = \rowsp([I_m \ A])$ and $\Lambda_{\mathbf{v}}(B) = \rowsp([I_m \ B])$, 
both subspaces have dimension $m$, and
\begin{align*} 
\dim(\Lambda_{\mathbf{u}}(A) \cap \Lambda_{\mathbf{u}}(B))& = \dim(\Lambda_{\mathbf{v}}(A) \cap \Lambda_{\mathbf{v}}(B)) \\
& = 2m - \dim(\Lambda_{\mathbf{v}}(A) + \Lambda_{\mathbf{v}}(B)) \\
&= 2m - \dim\big(\rowsp([I_m \ A]) + \rowsp([I_m \ B])\big) \\
&= 2m - \rk\begin{bmatrix}
I_m & A \\
I_m & B
\end{bmatrix} \\
&= 2m - \rk\begin{bmatrix}
I_m & A \\
0 & A-B
\end{bmatrix} \\
&= 2m - \big(m + \rk(A-B)\big) \\
&= m - \rk(A-B),
\end{align*}
as claimed.
\end{proof}

We are now ready to describe a colouring of $J_q(n,m,t)$.

\subsection{Colouring subspaces lifted from an MRD code }\label{sec:MRDidea}

We start by assuming $n \geq 2m$. 
Let $\CC$ be an MRD code in $M_{m\times n-m}(\F_q)$ with minimum distance 
$d(\CC) = m - t + 1$.  
Note that $n-m \geq m$ and by \Cref{thm:existenceMRD}, such a code always exists.  
Moreover,
\[
|\CC| = q^{(n-m)(m-(m-t+1)+1)} = q^{(n-m)t}.
\]

Consider the quotient space $M_{m \times n-m}(\F_q)/\mathcal{C}$ and let $\{ D_i: 0 \leq i \leq  q^{(n-m)(m-t)} - 1 \}$ be any set of representatives of distinct cosets for this quotient space. Hence
\[
\frac{M_{m \times n-m}(\F_q)}{\CC} 
= \{D_i + \CC : 0 \leq i \leq q^{(n-m)(m-t)} - 1 \},
\]
Note that 
\begin{equation} \label{eq:disjointproperties}
(D_i + \CC) \cap (D_j + \CC) = \emptyset
\end{equation}
for all $i \neq j$, and each coset satisfies
\begin{equation} \label{eq:sizecosets}
|D_i + \CC| = |\CC| = q^{(n-m)t}.
\end{equation}

Let $\uu \in \F_2^{n}$ be a binary vector with Hamming weight $m$. For every $i \in \{0,\ldots,q^{(n-m)(m-t)}-1\}$ define
\begin{equation} \label{eq:liftingcosets}
\mathcal{S}_{\uu,i} := \Lambda_{\uu}(D_i+\CC) = \{\Lambda_{\uu}(D_i+A) : A \in \CC\}.
\end{equation}

\begin{prop}\label{prop:disjoint}
For each $i \in \{0,\ldots,q^{(n-m)(m-t)}-1\}$, the set $\mathcal{S}_{\uu,i}$ consists of at most $q^{(n-m)t}$ distinct 
$m$-dimensional subspaces of $\F_q^{n}$ with the property that 
\[
\dim(U \cap U') \leq t-1\]
for all $U,U' \in \mathcal{S}_{\uu,i}$, with $U \neq U'$.
Moreover, $\mathcal{S}_{\uu,i} \cap \mathcal{S}_{\uu,j} = \emptyset$, for all $i \neq j$. 
\end{prop}

\begin{proof}
By Eq. \eqref{eq:sizecosets}, we have $|D_i + \CC| = |\CC| = q^{(n-m)t}$. Since we restrict the domain of $L_u$ to only those matrices $A$ for which $L_u(A)$ is in RREF, it follows that $|\mathcal{S}_{u, i}| \leq |D_i + \CC| = q^{(n-m)t}$.
Moreover, by Eq. \eqref{eq:disjointproperties} and the injectivity of the mapping $\Lambda_{\mathbf{u}}$, 
the families $\mathcal{S}_{\mathbf{u},i}$ are pairwise disjoint. It remains to show that for $U,U' \in \mathcal{S}_{\uu,i}$ with $U \neq U'$, one has 
$\dim(U \cap U') \leq t-1$.  
Indeed, if $U = \Lambda_{\uu}(D_i+A)$ and $U' = \Lambda_{\uu}(D_i+A')$ with $A,A' \in \CC$, then by \Cref{lm:intersectionlifting},
\[
\dim(U \cap U') = m - \rk(A-A').
\]
Since $A \neq A'$ and $\CC$ has minimum distance $d(\CC) = m-t+1$, we obtain $\rk(A-A') \geq m-t+1$, hence $\dim(U \cap U') \leq t-1$, as claimed.
\end{proof}

\subsection{The colouring of all $m$-spaces}\label{Subs:extendedcolouring}

Let $J_q(n, m,t)$ be the power of the Grassmann graph as above. 
Following the notation of \cite{etzion2009error}, to any $m$-dimensional subspace $S$ of $\F_q^{n}$ we can associate an \emph{identifying vector} $\vv(S)$. In particular, let $E(S)$ denote the unique RREF matrix for which rowsp$(E(S)) = S$. Then the identifying vector $\vv(S)$ is a binary vector of length $n$ and weight $m$, where the ones in $\vv(S)$ correspond to the positions (columns) of $E(S)$ in which there is a leading 1.

As an example, let $m=3$, $n=6$ and suppose the subspace $S$ of $\F_q^6$ is given by the row space of 
$$
E(S) = \begin{bmatrix}
    1 & 0 & * & 0 & * & * \\
    0 & 1 & * & 0 & * & * \\
    0 & 0 & 0 & 1 & * & *
\end{bmatrix}
$$
Then $\vv(S) = (1, 1, 0, 1, 0, 0)$, since the leading 1's appear in positions 1, 2, and 4.

\noindent
For two $m$-dimensional subspaces $S,T$ of $\F_q^{n}$, by using Eq. \eqref{eq:generalrelationdHammingshur}, 
the Hamming distance between their respective identifying vectors satisfies
\begin{equation} \label{eq:relationdHammingshur}
\dH(\mathbf{v}(S),\mathbf{v}(T)) 
= \wH(\mathbf{v}(S)) + \wH(\mathbf{v}(T)) 
- 2\wH(\mathbf{v}(S)\star \mathbf{v}(T)) 
= 2m - 2\wH(\mathbf{v}(S)\star \mathbf{v}(T)),
\end{equation}
where $\wH(\cdot)$ denotes the Hamming weight.

Assume that $n \geq 2m$. Partition the identifying vectors, which are the binary vectors of length $n$ and weight $m$, into parts such that for any two vectors $\mathbf{v}=(v_1,\ldots,v_{n})$ and $\mathbf{w}=(w_1,\ldots,w_{n})$ in the same part, we have that \[\wH(\mathbf{v} \star \mathbf{w}) \leq t-1.\]
This corresponds to a partition of the $m$-subsets of a $n$-set into parts such that any two $m$-subsets in the same part intersect in at most $t-1$ elements. Hence, the least number of partition classes is equal to the chromatic number of the Johnson graph $J(n , m, t)$. 

Take such a partition with $\chi(J(n,m,t))$ classes of the identifying vectors, and denote them by $\{ I_1^c, I_2^c, \dots, I_{\chi(J(n, m, t))}^c \}$. Then we partition all $m$-spaces $S$ of $\F_q^n$ in $\chi(J(n, m, t))$ parts, where to $m$-spaces lie in the same part if their identifying vectors belong to the same class $I_j^c$. By a slight abuse of notation, we will denote by $I_j$ the set of all subspaces whose identifying vector belongs to the set $I_j^c$.

We now assign a colour to each subspace.
The colouring method used will be the same for each part $I_j$, however, the colours used for colouring one part $I_j$ are all distinct from those used for any other parts.
Hence, it suffices to describe how we colour the subspaces within a single class $I_j$.

Consider the class of identifying vectors $I_j^c=\{\vv_1,\vv_2,\ldots\}$. For any identifying vector $\vv_h \in I_j^c$ we colour all subspaces $S$ with $\vv(S)=\vv_h$ using the MRD code argument from Section \ref{sec:MRDidea}. Precisely, let $\mathbf{v}_h \in I_j^c$. 
We colour all subspaces $S$ with $\mathbf{v}(S) = \mathbf{v}_h$ as follows.  
For each fixed index $i \in \{0,\ldots,q^{(n-m)(m-t)}-1\}$, we assign the same colour to all subspaces $S \in \mathcal{S}_{\mathbf{v}_h,i}$, 
where $\mathcal{S}_{\mathbf{v}_h,i}$ is defined as in Eq. \eqref{eq:liftingcosets}.  
Different indices $i$ correspond to different colours; that is, 
if $i \neq j$, then the subspaces in $\mathcal{S}_{\mathbf{v}_h,i}$ and the subspaces in $\mathcal{S}_{\mathbf{v}_h,j}$ 
receive distinct colours.

We therefore use $q^{(n-m)(m-t)}$ colours for these subspaces with $\vv(S)=\vv_h$. We repeat this colouring for all spaces $T$ with a fixed identifying vector belonging to $I_j^c$. We arbitrarily reuse the colours that were used to colour all spaces $S$ with $\vv(S) = \vv_h$. Hence, we use $q^{(n-m)(m-t)}$ colours to colour all subspaces in $I_j$.

We can do this for each of the $\chi(J(n, m, t))$ sets $I_j$, and hence this gives a colouring of the graph $J_q(n,m,t)$ with $\chi(J(n, m, t))q^{(n-m)(m-t)}$ colours. 

We now prove that this colouring is a good and valid colouring of  $J_q(n,m,t)$. For this, we first need the following lemma. 

\begin{lem}\label{lem:sameclass}
    Suppose $S$ and $T$ are two subspaces in $I_j$ for some $j$.  Then $\dim(S\cap T)\leq t-1$ or $\vv(S)=\vv(T)$.
\end{lem}
\begin{proof}
    Suppose $S, T \in I_j$, and  $\vv(S)\neq \vv(T)$. By \cite[Lemma 2]{etzion2009error} and by using Eq. \eqref{eq:relationdHammingshur}, for any two $m$-dimensional subspaces $S$ and $T$ it follows that
\[
  2\dim(S \cap T) \leq 2m - d_H(\vv(S), \vv(T)) 
  = 2m - (2m - 2\wH(\vv(S)\star \vv(T))) 
  = 2\wH(\vv(S)\star \vv(T)).
\]
Hence,
\begin{equation} \label{eq:dimensionidentifyingvector}
  \dim(S \cap T) \leq  \wH(\vv(S)\star \vv(T)) \leq t-1,
\end{equation}
in which the last inequality follows from the definition of $I_j$.
\end{proof}

\begin{thm} \label{th:maincolouringnbig}
    Assume that $n \geq 2m$. The colouring described above is a valid colouring and hence \[\chi(J_q(n, m, t))\leq \chi(J(n, m, t))q^{(n-m)(m-t)}\]
\end{thm}
\begin{proof}
    It is immediately clear that every subspace belongs to a certain class $I_j$, and that therefore every subspace has been assigned a colour in the above colouring.
We must therefore still show that two subspaces with the same colour have a subspace of at most dimension $t-1$ in common.
Suppose that $S$ and $T$ are subspaces with the same colour $c$, and assume by contradiction that $\dim(S\cap T)\geq t$. Since two subspaces from two different sets $I_j\neq I_i$ have different colours by definition, we know that there exists a $j$ such that $S,T\in I_j$. Moreover, Lemma \ref{lem:sameclass} implies that $\mathbf{v}(S) = \mathbf{v}(T)$; hence $S$ and $T$ are coloured according to the same MRD code idea.  
Let $\mathbf{u} := \mathbf{v}(S) = \mathbf{v}(T)$.  
Then $S \in \mathcal{S}_{\mathbf{u},h}$ and $T \in \mathcal{S}_{\mathbf{u},k}$ for some $h,k \in \{0,\ldots,q^{(n-m)(m-t)}-1\}$.  
Since $S$ and $T$ have the same colour, it follows that $h = k$.  However, by Proposition \ref{prop:disjoint}, any two distinct subspaces within the same family $\mathcal{S}_{\mathbf{u},h}$ intersect in dimension at most $t-1$.  
This contradicts our assumption that $\dim(S \cap T) \ge t$.  
Therefore, two subspaces with the same colour cannot intersect in dimension greater than $t-1$, 
and the colouring is proper. This completes the proof 
\end{proof}

Using the bound on the chromatic number of the powers of the Johnson graph described in Eq. \eqref{eq:ordergeneralizedJohnson}, 
we obtain the following result.

\begin{cor}\label{cor1}
     Assume that $n \geq 2m$. We have \[\chi(J_q(n, m, t))=  O(n^{m-t}q^{(n-m)(m-t)})\]
\end{cor}


We now study the case $m < n < 2m$.  
In this range, an upper bound on the chromatic number of $J_q(n,m,t)$ can be obtained by duality.  
The following result holds.

\begin{thm} \label{th:maincolouringnsmall}
    Assume that $m < n < 2m$. Then 
    \[
    \chi  \left(J_q(n, m, t)\right) 
      \leq  
    \chi  \left(J(n, m, t)\right)  q^{m(m-t)}.
    \]
\end{thm}

\begin{proof}
By (\ref{Isomorphism}), the graph $J_q(n,m,t)$ is isomorphic to  
$J_q(n,n-m,n-2m+t)$.  
Hence,
\[
\chi(J_q(n,m,t)) = \chi(J_q(n,n-m,n-2m+t)).
\]
Since in this case $n > 2(n-m)$, we may apply Theorem \ref{th:maincolouringnbig} to obtain
\[
\chi(J_q(n,m,t)) 
= \chi(J_q(n,n-m,n-2m+t))
\leq 
\chi(J(n,n-m,n-2m+t))  q^{m(m-t)}.
\]
Finally, it is easy to observe that 
\[
\chi(J(n,m,t)) = \chi(J(n,n-m,n-2m+t)),
\]
which completes the proof.
\end{proof}

Again, by using the bound on the chromatic number of the powers of the Johnson graph described in Eq. \eqref{eq:ordergeneralizedJohnson}, 
we get the following result.

\begin{cor}\label{cor2}
     Assume that $m < n < 2m$. Then  \[\chi(J_q(n, m, t))=  O(n^{m-t}q^{m(m-t)}).\]
\end{cor}

\section{Concluding Remarks}

In this paper, we studied the chromatic number of the Grassmann graph and of its powers,
introducing a new colouring technique based on MRD codes. Our bounds show that the asymptotic growth of $\chi(J_q(n,m,t))$ is governed by the size of
the largest natural cliques in the graph, up to polynomial factors arising from the
corresponding Johnson graph.
In particular, for fixed parameters $n$, $m$, and $t$, our results asymptotically determine
the behaviour of $\chi(J_q(n,m,t))$ as a function of $q$.
More precisely, as $q \to \infty$,
\[
\chi(J_q(n,m,t)) = \Theta\!\left(q^{(m-t)\min\{n-m,m\}}\right).
\]

On the other hand, when $q$ is fixed and $n$ grows, the exponential term becomes constant,
and the dominant contribution is polynomial in $n$, inherited from the chromatic behaviour
of the Johnson graph or its powers. Therefore, one interesting direction for future research is to investigate whether
the polynomial factor \(n^{m-t}\) appearing in Corollary \ref{cor1} and Corollary \ref{cor2} can be improved,
or possibly even eliminated, in certain regimes of the parameters. When $n$ is large enough compared to $q$, the result in \cite{dhaeseleer2025chromatic} still yields the best-known upper bound for the chromatic number of the standard Grassmann graphs (i.e. $J_q(n, m, m-1)$). It would be of interest to see if sharper results can be obtained in special cases, where additional geometric interpretations may be exploited.

\bibliographystyle{abbrv}
\bibliography{sample.bib}

\end{document}